\documentclass[a4paper, 12pt]{amsart}

\usepackage{amssymb, amsmath}
\usepackage{color}
\usepackage{marginnote}

\renewcommand{\marginpar}[1]{}

\numberwithin{equation}{section}

\theoremstyle{plain}
\newtheorem{proposition}{Proposition}[section]
\newtheorem{theorem}[proposition]{Theorem}
\newtheorem{lemma}[proposition]{Lemma}

\newtheorem{corollary}[proposition]{Corollary}

\theoremstyle{definition}

\newtheorem{definition}[proposition]{Definition}

\newcommand{\C}{\mathbb{C}}
\newcommand{\N}{\mathbb{N}}
\newcommand{\Z}{\mathbb{Z}}
\newcommand{\R}{\mathbb{R}}
\renewcommand{\S}{\mathbb{S}}

\DeclareMathOperator{\lcm}{lcm}

\newcommand{\Cal}{\mathcal}

\def \<{\langle}
\def \>{\rangle}

\usepackage{mathtools}

\DeclarePairedDelimiter\floor{\lfloor}{\rfloor}

\title[Beatty Sequences]
{Expansions of the Group of Integers by Beatty Sequences}

\begin{document}

\author {Ayhan G\"unayd\i n}

\address{Department of Mathematics, Bo\u{g}azi\c{c}i University, Bebek, Istanbul, Turkey}

\email{ayhan.gunaydin@boun.edu.tr}

\author {Melissa \"Ozsahakyan}

\address{Department of Mathematics, Bo\u{g}azi\c{c}i University, Bebek, Istanbul, Turkey}

\email{assilem89@gmail.com}

\begin{abstract}
We study the model theoretic structure $(\Z,+,P_r)$ where $r>1$ is an irrational number and the elements of $P_r$ are of the form $\floor{nr}$ for some $n\in\Z\setminus\{0\}$. We axiomatize of this structure and prove a quantifier elimination result. As a consequence, we get that definable subsets are not sparse unless they are finite. We also prove that there are no reducts of this structure expanding $(\Z,+)$.
\end{abstract}

\bibliographystyle{plain}

\maketitle

\section{Introduction}
We investigate the expansion of the abelian group of integers by the following subset:

 \[
  P_r:=\Big\{\floor{nr}:n\in\Z\setminus\{0\}\Big\},
 \]
where $r>1$ is an irrational number.
The number $0$ is taken out for some technical reasons and of course it has no effect on our results. There is a considerable difference between this expansion and the expansion by $P_r^+:=P_r\cap \N$, because $(\Z,+,P_r^+)$ defines the ordering of $\Z$, whereas we show in Corollary~\ref{ordering_cor} that $(\Z,+,P_r)$ does not define the ordering. We also know that $P_r$ is not definable in $(\Z,+,<)$, since it is shown in \cite{Conant_no_intermediate} that the only reduct of $(\Z,+,<)$ defining addition is $(\Z,+)$.

There have been some work on expansions of $(\Z,+)$ by a predicate by Poizat , Palac{\'i}n-Sklinos, and Lambotte-Point (\cite{Poizat_Exp_Integers,Palacin-Sklinos,Lambotte-Point}). In all these, the predicate is \emph{sparse}, with certain different but similar meanings of the word sparse.  Later those results were generalized by Conant (in \cite{Conant_multiplicative}), who proved that the expansion of $(\Z,+)$ by an infinite subset $A$ of a submonoid of $(\N_+,\cdot)$ is superstable of $U$-rank $\omega$. Hence in that setting, sparsity comes from the multiplicative structure. 

Our study of the expansion $(\Z,+,P_r)$ is complimentary to the work mentioned above, as the set $P_r$ is certainly not sparse: If $k\in P_r$, then the next element in $P_r$ is either $k+\floor{r}$ or $k+\floor{r}+1$. (This is also the reason of $(\Z,+,P_r^+)$ defining the ordering.)  Actually, we prove that no infinite subset of $\Z$ definable in $(\Z,+,P_r)$ is sparse in the following sense:

\begin{theorem}\label{th_1}
Let $X\subseteq \Z$ be infinite and definable in $(\Z,+,P_r)$. Then there is $N=N(X)\in\N_+$ such that for any $x\in X$, one of the integers $x+1,x+2,\dots,x+N$ is in $X$.
\end{theorem}

The proof of this result goes through  a quantifier elimination result. In order to state that, let $L_\pm$ be the extension of the language $\{+,-,0,1\}$ of abelian groups with a distinguished element by unary predicate symbols $D_{m,+}$ and $D_{m,-}$ for each $m\geq 1$.   We interpret the new symbols in $\Z$ as follows:
 \[
 D_{m,+}^\Z:=\{x\in\Z: x=my \text{ for some }y\in P_r\},
 \]
and
 \[
 D_{m,-}^\Z:=\{x\in\Z: x=my \text{ for some }y\notin P_r\}.
 \]
 Note that $D_{1,+}^\Z=P_r$ and that $D_{m,+}^\Z\cup D_{m,-}^\Z=m\Z$.

\begin{theorem}\label{th_2}
The structure $\Big(\Z,+,0,1,(D_{m,+}^\Z)_{m\geq 1},(D_{m,-}^\Z)_{m\geq 1}\Big)$ has quantifier elimination.
\end{theorem}

A particular formula with quantifiers is 
 \[
 \exists y\left(\bigwedge_{i\in I}x_i+k_iy\in P\wedge\bigwedge_{j\notin I}x_j+k_jy\notin P\right),
 \]
where $k_1,\dots,k_n\in\Z$ and $I\subseteq\{1,\dots,n\}$.  So we have a quantifier-free $L_\pm$-formula $\psi_{\vec k,I}(\vec x)$ such that 
 \[
(\star) \forall x_1\cdots\forall x_n\left(\exists y\left(\bigwedge_{i\in I}x_i+k_iy\in P\wedge\bigwedge_{j\notin I}x_j+k_jy\notin P\right)\leftrightarrow\psi_{\vec k,I}(\vec x)\right).
 \]
holds in $\Big(\Z,+,0,1,(D_{m,+}^\Z)_{m\geq 1},(D_{m,-}^\Z)_{m\geq 1}\Big)$. As a matter of fact, $\psi_{\vec k,I}$ can be chosen to be a formula in the language $L_P=\{+,-,0,1,P\}$.

\medskip\noindent
With this notation at hand, we have the following axiomatization.

\begin{theorem}\label{th_3}
Let $\Cal M=\big(M,+,-,0,1,P^\Cal M\big)$ be an $L_P$-structure. Then $\Cal M$ is elementarily equivalent to $(\Z,+,-,0,1,P_r)$ if and only if the following hold
 \begin{enumerate}
  \item $(M,+,-,0,1)\equiv(\Z,+,-,0,1)$,
  \item for every $k_1,\dots,k_n\in\Z$ and $I\subseteq\{1,\dots,n\}$ the sentence $(\star)$ holds in $\Cal M$,
  \item $k\in P_r$ if  and only if $k\in P^\Cal M$ for every $k\in\Z$. 
 \end{enumerate}
\end{theorem}

\medskip\noindent
The technical parts of the proofs are done in an isomorphic structure: Let 
 \[
  \Gamma_r:=\left\{\exp(\frac{n2\pi i}{r})\in\C:n\in\Z \right\}.
 \] 
So $\Gamma_r=h(\Z)$ where $h$ is the group isomorphism sending $n$ to $\exp(\frac{n2\pi i}{r})$. Being a subgroup of the unit circle, $\Gamma_r$ has an \emph{orientation} on it; the precise definition is given in the next section. Then the image of $P_r$ under $h$ becomes an \emph{orientation interval}; see Lemma~\ref{sturm-vs-beatty1}. This makes it easier to work in $\Gamma_r$ and the notations get simpler. For this reason, in Section~\ref{circle_section}, we recall some facts about the circle and its subgroups. 

\medskip\noindent
In Section~\ref{beatty_section}, we introduce Beatty Sequences and prove a few results about them to be used in the model theoretic arguments. 

\medskip\noindent
A back-and-forth system constructed in Section~\ref{back_and_forth_section} is used to prove the theorems mentioned above in the rest of  that section and Section~\ref{definable_sets_section}.

\medskip\noindent
The paper \cite{Tran-Walsberg_expansions_of_Z} by  Tran-Walsberg has quite a bit of overlap with our work. The authors consider $\Z$ equipped with a `cyclic ordering'; which in turn is the same as considering an infinite cyclic subgroup of the circle equipped with the orientation. We elaborate on this in  Section~\ref{definable_sets_section}. 

\medskip\noindent
In the last section, we prove a result analogous to the main result of \cite{Conant_no_intermediate}: there are no intermediate structures between $(\Z,+)$ and $(\Z,+,P_r)$.

\medskip\noindent
\emph{Notations and Conventions.} The set $\N$ of natural numbers contains $0$ and $\N_+=\N\setminus\{0\}$. We let the letters $m,n,k,l$ vary in $\Z$, and if $m$ is in $\N$ (or $\N_+$), we simply write $m\geq 0$ (or $m>0$).

\medskip\noindent
For $n>0$, we denote the set $\{1,\dots,n\}$ as $[n]$. 

\medskip\noindent
For a real number $a$, we use the notation $\floor{a}$ for the largest integer smaller than or equal to $a$ and $\{a\}$ denotes the difference $a-\floor{a}$. (There will not be any occasions where this could be confused with the singleton containing $a$.)

\section{The Circle and Its Subgroups}\label{circle_section}
In the next section, we work with an infinite cyclic subgroup of the unit circle $\S:=\{\beta\in\C:|\beta|=1\}$. One may study such a group in the generality of \emph{oriented abelian groups} as defined in \cite{thesis_ayhan}, however there is no need to do so for our purposes. 

\medskip\noindent
Here we recall some generalities about $\S$ and at the end we say a few words about its subgroups.
\medskip\noindent
Let
 \[
 e:\R\to\S,\quad e(x):=\exp(2\pi i x).
 \]
 This is a surjective group homomorphism with kernel $\Z$. 
 
 \medskip\noindent
We equip $\S$ with the \emph{counter-clockwise orientation}: Given  $\alpha,\beta,\gamma\in\S$, the relation $\Cal O(\alpha,\beta,\gamma)$ holds if and only if there are $x,y,z\in\R$ such that $\alpha=e(x),\beta=e(y),\gamma=e(z)$, $x<y<z$, and $z-x<1$.

\medskip\noindent
 If we fix $\alpha\in\S$, then we get a linear ordering $\Cal O_{\alpha}(\cdot,\cdot):=\Cal O(\alpha,\cdot,\cdot)$ on $\S\setminus\{\alpha\}$. For $\alpha=1$, we denote the ordering of $\S\setminus\{1\}$ by the usual ordering sign:
  \[
  \beta<\gamma\Longleftrightarrow \Cal O(1,\beta,\gamma).
  \]
  
 \noindent
Clearly, if $\alpha<\beta$ and $\beta<\gamma$, then $\Cal O(\alpha,\beta,\gamma)$. So there is no harm in writing 
 \[
 \alpha_1<\alpha_2<\cdots<\alpha_{n-1}<\alpha_n
 \]
when $\alpha_1<\alpha_2,\alpha_2<\alpha_3,\dots,\alpha_{n-1}<\alpha_n$.

\medskip\noindent
We extend the definition of $<$ to all of $\S$ by setting $1<\alpha$ for $\alpha\neq 1$. 

\medskip\noindent
The relation of orientation and the group operation is as follows:
 \[
 \Cal O(\alpha,\beta,\gamma)\Longleftrightarrow\Cal O(\alpha\delta,\beta\delta,\gamma\delta)
 \]
and
 \[
 \Cal O(\alpha,\beta,\gamma)\Longleftrightarrow\Cal O(\gamma^{-1},\beta,^{-1}\alpha^{-1})
 \]
for every $\alpha,\beta,\gamma,\delta\in\S$  .

\medskip\noindent
The circle $\S$ has the topology induced by the Euclidean topology on $\C$ and a basis for this topology consists of \emph{orientation intervals}: Given $\alpha,\beta\in\S$ we define the orientation interval determined by $\alpha$ and $\beta$ to be
 \[
 (\alpha,\beta):=\big\{\gamma\in\S:\Cal O(\alpha,\gamma,\beta)\big\}.
 \]
We do not assume $\alpha<\beta$ for this definition. So both $(\alpha,\beta)$ and $(\beta,\alpha)$ are orientation intervals and they are disjoint. Such an interval is empty only when $\alpha=\beta$. 

\medskip\noindent
We define the \emph{length} of an orientation interval $I=(\alpha,\beta)$ to be 
 \[l(I):=\beta\alpha^{-1}.\]
 So if $\alpha=e(a)$ and $\beta=e(b)$, then the length of $(\alpha,\beta)$ is $e(b-a)$. In particular, the length of an orientation interval is $1$ if and only if it is empty.

\medskip\noindent
 Below, we use the word \emph{interval} to mean \emph{orientation interval}.  We also use the notations $[\alpha,\beta)$, $(\alpha,\beta]$, and  $[\alpha,\beta]$ with the obvious meanings, and we call them \emph{intervals} as well.


\begin{proposition}\label{intersection_of_intervals}
Let $\alpha_1,\dots,\alpha_m,\beta_1,\dots,\beta_m$ be distinct elements of $\S$. Then the following conditions are equivalent:
 \begin{enumerate}
  \item $\displaystyle{\bigcap_{i=1}^m} (\alpha_i,\beta_i)\neq\emptyset$.
  \item there is $i_0\in[m]$ such that $\alpha_{i_0}\in (\alpha_i,\beta_i)$ for every $i\neq i_0$. 
  \item there is $j_0\in[m]$ such that $\beta_{j_0}\in (\alpha_i,\beta_i)$ for every $i\neq j_0$. 
 \end{enumerate}
Moreover, if one of these conditions hold, then 
 \[
  \bigcap_{i=1}^m (\alpha_i,\beta_i)=(\alpha_{i_0},\beta_{j_0}).
 \] 
\end{proposition}

\begin{proof}
%
Suppose $\gamma\in(\alpha_i,\beta_i)$ for every $i\in[m]$. Then $\Cal O_\gamma(\beta_i,\alpha_i)$ for each $i\in[m]$. Choose $i_0$ such that $\alpha_{i_0}$ is maximal among $\alpha_i$ with respect to $\Cal O_\gamma$. Clearly, $\alpha_{i_0}\in(\alpha_i,\beta_i)$ for each $i\neq i_0$. Similarly, taking $\beta_{j_0}$ to be minimum among $\beta_i$ with respect to $\Cal O_\gamma$ we see that $\beta_{j_0}\in(\alpha_i,\beta_i)$ for each $i\neq j_0$. So the first condition implies the others.

\medskip\noindent
Conversely, assume that $\alpha_{i_0}\in(\alpha_i,\beta_i)$ for each $i\neq i_0$. Then 
 \[(\alpha_{i_0},\beta_{j_0})\subseteq\bigcap_{i=1}^m (\alpha_i,\beta_i),\]
where $\beta_{j_0}$ is minimum among $\beta_i$ with respect to $\Cal O_{\alpha_{i_0}}$. 

\medskip\noindent
If $\beta_{j_0}\in(\alpha_i,\beta_i)$ for each $i\neq i_0$. Then we take $\alpha_{i_0}$ to be maximum among $\alpha_i$ with respect to $\Cal O_{\beta{j_0}}$ in order to get  
 \[(\alpha_{i_0},\beta_{j_0})\subseteq\bigcap_{i=1}^m (\alpha_i,\beta_i).\]

\medskip\noindent
The last sentence of the proposition follows from the rest of the proof above.
\end{proof}

\medskip\noindent
For $k\in\N_+$, let $\zeta_k$ denote the primitive $k^{\text{th}}$ root of unity $e(\frac{1}{k})$.

\begin{proposition}\label{kthpower}
Let $\alpha,\beta,\gamma\in\S$ and $k\in\N_+$ such that $l\big((\alpha,\gamma)\big)<\zeta_k$. Then $\beta^k\in(\alpha^k,\gamma^k)$ if and only if $\beta\in(\zeta_k^s\alpha,\zeta_k^s\gamma)$ for some $s\in\Z$.
\end{proposition}

\begin{proof}
Note that if $\alpha=\gamma$, then the result is clear; so we assume $\alpha\neq\gamma$.  We may also assume that $\alpha<\gamma$; if that is not the case, then multiply $\alpha$ and $\gamma$ by $\zeta_k$.  So let $0\leq a< c<1$ be such that $\alpha=e(a)$ and $\gamma=e(c)$. Note that $c-a<\frac{1}{k}$ by assumption. Also let $\beta=e(b)$ for some $0\leq b<1$. 

\medskip\noindent
Suppose $\beta^k\in(\alpha^k,\gamma^k)$ and take $b'$ such that $\beta^k=e(b')$ and $ka<b'<kc$. Then $s:=kb-b'\in\Z$ and we get 
 \[
 ka+s<kb<kc+s.
 \]
Dividing by $k$, we have
 \[
 a+\frac{s}{k}<b<c+\frac{s}{k}.
 \]
 Applying $e$, we get that $\beta\in(\zeta_k^s\alpha,\zeta_k^s\gamma)$.
 
\medskip\noindent
Conversely, let $\beta\in(\zeta_k^s\alpha,\zeta_k^s\gamma)$ for some $s\in\Z$. In other words, $\beta\zeta_k^{-s}\in(\alpha,\gamma)$. Take $b''$ such that $\beta\zeta_k^{-s}=e(b'')$ and $a<b''<c$. Then $t:=b-b''-\frac{s}{k}\in\Z$ and
 \[
 a+t<b-\frac{s}{k}<c+t.
 \]
Now multiplying by $k$ and applying $e$ we get $\beta^k\in(\alpha^k,\gamma^k)$.
\end{proof}

\medskip\noindent
Note that the $s$ in this proposition can be chosen among $0,1,\dots,k-1$.

\medskip\noindent
Given $\alpha=e(x)$ with $0\leq x<1$ and $k\in\N_+$, we let $\alpha^{1/k}$ denote $e(\frac{x}{k})$; so $\alpha^{1/k}$ is the $k^{\text{th}}$ root of $\alpha$ with the smallest argument. Note that $1^{1/k}=1$ and that for $x\in\R$ we have $e(x)^{1/k}=e(\frac{x}{k})\zeta_k^{-\floor{x}}$. In particular, if $\alpha=e(a)$, then $(\alpha^k)^{1/k}=\alpha\zeta_k^{-\floor{ak}}$. However, we always have $(\alpha^{1/k})^k=\alpha$. The following observation will be useful.

\begin{lemma}\label{kthroot_roots_of_unity}
Let $\alpha,\beta\in\S$, $k,l\in\Z$. Suppose that $\alpha^k=\beta$. Then $\alpha=\beta^{1/k}\zeta_k^{l}$ if and only if $\alpha\in(\zeta_k^l,\zeta_k^{l+1})$.
\end{lemma}

\begin{proof}
Clear.
\end{proof}

\medskip\noindent
Using this new notation, the previous proposition has the following consequences.

\begin{corollary}\label{kthroot1}
Let $\alpha,\beta,\gamma\in\S$ and $k\in\N_+$.  Suppose that $\gamma\not<\alpha$. Then $\beta^k\in(\alpha,\gamma)$ if and only if there is $s\in\{0,1,\dots,k-1\}$ such that $\beta\in\big(\zeta_k^s\alpha^{1/k},\zeta_k^s\gamma^{1/k}\big)$.
\end{corollary}

%
%
%
%
 
 \begin{corollary}\label{kthroot2}
Let $\alpha,\beta,\gamma\in\S$ and $k\in\N_+$.  Suppose that $\gamma< \alpha$. Then $\beta^k\in(\alpha,\gamma)$ if and only if there is $s\in\{0,1,\dots,k-1\}$ such that $\beta\in\big(\zeta_k^s\alpha^{1/k},\zeta_k^{s+1}\gamma^{1/k}\big)$.
\end{corollary}

%
%


\medskip\noindent
{\bf Regularly Dense Groups}. Let $\Gamma\leq\S$. Then $\Gamma$ is either finite or dense in $\S$. When it is finite, it consists of $m^{\text{th}}$ roots of unity for some $m>0$. When $\Gamma$ is dense, it is indeed \emph{regularly dense} in the following sense.

\begin{proposition}\label{regularly_dense_prop}
Let $\Gamma\leq\S$ be infinite. Then for any distinct $\alpha,\beta\in\Gamma$ and prime $p$, there is $\gamma\in\Gamma$ such that $\gamma^p\in(\alpha,\beta)$. 
\end{proposition}

\medskip\noindent
For the proof of this, we refer the reader to Definition 8.1.7 in \cite{thesis_ayhan} and the remark succeeding it. Note that the conclusion of the proposition above is slightly different than the original definition of regularly dense, but it is easy see that they are indeed equivalent. It follows that for any $n>0$ and distinct $\alpha,\beta\in\Gamma$, there is $\gamma\in\Gamma$ such that $\gamma^n\in(\alpha,\beta)$.

\section{Beatty Sequences}\label{beatty_section}
Let $r>1$ be an irrational number. The \emph{Beatty Sequence generated by $r$} is $\Cal B_r=\left(\floor{mr}\right)_{m>0}$; we put $b_m=\floor{mr}$.  This is an increasing sequence and we let $P_r^+$ denote the set whose elements are the terms of $\Cal B_r$. 

\medskip\noindent
A related sequence is $\Cal S_r=\left(\floor{\frac{n+1}{r}}-\floor{\frac{n}{r}}\right)_{n>0}$; we put $s_n=\floor{\frac{n+1}{r}}-\floor{\frac{n}{r}}$. Note that $s_n\in\{0,1\}$ for each $n>0$. Actually, it is better to think of $\Cal S_r$ as an infinite word in the alphabet $\{0,1\}$. As such, it is called \emph{the Characteristic Sturmian Word of Slope $\frac{1}{r}$}. It has the property that for every $m$, it has exactly $m+1$ many different subwords of length $m$. 

\medskip\noindent
Both Beatty Sequences and Sturmian Words have rich theories that we do not get into here, and we refer the interested reader to \cite{auto}. We only need the following connection between $\Cal B_r$ and $\Cal S_r$ which is Lemma 9.1.3 of \cite{auto}, but we include a proof for completeness.

\begin{lemma}\label{sturm-vs-beatty1}
Let $n\in\N_+$. Then  $n\in P_r^+$ if and only if $s_n=1$.
\end{lemma}

\begin{proof}
Let $n\in P_r^+$. Then $n=\floor{kr}$ for some $k\in\N_+$. So $kr-1<n<kr$ and after diving by $r$ we have 
 \[
 k-\frac{1}{r}<\frac{n}{r}<k.
 \]
 Therefore $\floor{\frac{n}{r}}=k-1$ and $\floor{\frac{n+1}{r}}=k$. Thus $s_n=\floor{\frac{n+1}{r}}-\floor{\frac{n}{r}}=1$. As all the implications are reversible we get the desired result.
\end{proof}

\medskip\noindent
We would like to consider the negative elements as well; so we define
 \[
 P_r=\big\{\floor{nr}:n\in\Z\setminus\{0\}\big\}. 
 \]

 \medskip\noindent
For $m>0$, we have $-m\in P_r$ if and only if $m-1\in P_r^+$. So 
 \[P_r=P_r^+\cup(-P_r^+-1).\]
We also extend the definitions of $b_n$ and $s_n$ to all integers $n$.

\medskip\noindent
Lemma~\ref{sturm-vs-beatty1} is actually correct for all $n\in\Z$:
  \begin{equation}\label{sturm-vs-beatty2}
  n\in P_r\Longleftrightarrow s_n=1.
 \end{equation}

\medskip\noindent
It is easy to see that $s_n=1$ if and only if $\{\frac{n}{r}\}>1-\frac{1}{r}$. Putting this together with (\ref{sturm-vs-beatty2}), for every $n\in\Z$, we get 
 \begin{equation}\label{sturm-vs-beatty3}
  n\in P_r\Longleftrightarrow \big\{\frac{n}{r}\big\}>1-\frac{1}{r}.
 \end{equation}

\medskip\noindent
Since $r$ is irrational, the image of $\Z\frac{1}{r}$ under $e$ is not finite, hence it is a dense subgroup of $\S$. Let $\Gamma_r$ be that subgroup, and let $h$ denote the map $n\mapsto e(\frac{n}{r})$.  So we have an isomorphism of abelian groups with a distinguished element:
 \[
 h:\big(\Z,+,-,0,1\big)\simeq\big(\Gamma_r,\cdot,^{-1},1,h(1)\big).
 \]
 By (\ref{sturm-vs-beatty3}), the image of $P_r$ under $h$ is $\big(h(-1),1\big)\cap\Gamma_r$. Therefore expanding $\Z$ by $P_r$ is the same as expanding $\Gamma_r$ by $\big(h(-1),1\big)\cap\Gamma_r$.

\medskip\noindent
Using Proposition~\ref{kthpower} and its corollaries, 
we give a criterion for certain linear combinations of integers being in $P_r$ in terms of intervals in $\S$.
 
\begin{proposition}\label{main_prop}
 Let $k\in\N_+$ and $a,c\in\Z$. Then $a+kc\in P_r$ if and only if there is $s\in\Z$ such that 
 \[
 h(c)\in\big(h(-a-1)^{1/k}\zeta_k^s,h(-a)^{1/k}\zeta_k^{s+s_a}\big).
 \]
\end{proposition}

\begin{proof}
First, note that $a\in P_r$ if and only if $h(-a)<h(-a-1)$.

\medskip\noindent
By (\ref{sturm-vs-beatty3}) we have
 \[
 a+kc\in P_r \Longleftrightarrow h(a)h(c)^k\in \big(h(-1),1\big)\Longleftrightarrow h(c)^k\in\big(h(-a-1),h(-a)\big).
 \]
Now combining Corollaries \ref{kthroot1} and \ref{kthroot2} and using the first sentence of this proof, we get that
 \[
  h(c)^k\in\big(h(-a-1),h(-a)\big)
 \]
  if and only if there is $s\in\Z$ with  
\[
 h(c)\in\big(h(-a-1)^{1/k}\zeta_k^s,h(-a)^{1/k}\zeta_k^{s+s_a}\big).
\]
This  gives the desired equivalence.
\end{proof}

\begin{corollary}\label{main_cor}
  Let $k\in\N_+$ and $a,c\in\Z$. Then $a+kc\notin P_r$ if and only if there is $s\in\Z$ such that 
 \[
 h(c)\in\big[h(-a)^{1/k}\zeta_k^{s+s_a},h(-a-1)^{1/k}\zeta_k^{s+1}\big].
 \]
\end{corollary}

\begin{proof}
Clear from the previous proposition.
\end{proof}

\medskip\noindent
Next result will be useful in handling the cases when $k$ is negative.
\begin{lemma}\label{negations}
Let $a,c\in\Z$ and $k<0$. Then $a+kc\in P_r$ if and only if $-a-1-kc\in P_r$.
\end{lemma}

\begin{proof}
Clear from the fact that $\floor{-x}=-\floor{x}-1$ for $x\notin\Z$.
\end{proof}

\begin{definition}
For  $a\in\Z$, $k\in\N_+$, and $s\in\{0,1,\dots,k-1\}$, let 

 \[
  U_{a,k,s}:=\big(h(-a-1)^{1/k}\zeta_k^s,h(-a)^{1/k}\zeta_k^{s+s_a}\big),
 \]
 \[
  V_{a,k,s}:=\big[h(-a)^{1/k}\zeta_k^{s+s_a},h(-a-1)^{1/k}\zeta_k^{s+1}\big].
 \]
 Also let 
  \[
   U_{a,k}:=\bigcup_{s=0}^{k-1}U_{a,k,s}
\quad\text{ and }\quad
   V_{a,k}:=\bigcup_{s=0}^{k-1}V_{a,k,s}.
    \]  
 We extend the definitions to $k=0$ as follows:
 \[
U_{a,0}:=
  \begin{cases}
 \S & :\text{if }a\in P_r\\
 \emptyset  & :\text{if }a\notin P_r
  \end{cases} 
 \quad\text{ and }\quad
 V_{a,0}:=
  \begin{cases}
 \emptyset & :\text{if }a\in P_r\\
 \S & :\text{if }a\notin P_r
  \end{cases} 
 \]
  Finally, we let $\tilde V_{a,k}$ denote the interior of $V_{a,k}$.
\end{definition}

\medskip\noindent
With this notation in hand, Proposition~\ref{main_prop} and Corollary~\ref{main_cor} translate as follows: Given $k\in\N$ and $a,c\in\Z$ we have 
 \begin{equation}\label{5}
  a+kc\in P_r\Longleftrightarrow h(c)\in U_{a,k},
 \end{equation}
and
 \begin{equation}\label{6}
  a+kc\notin P_r\Longleftrightarrow h(c)\in V_{a,k}.
 \end{equation}
 
 \begin{lemma}\label{U's}
 Let $a,b\in\Z$ and $k,l\in\N_+$.  Suppose $g=\gcd(k,l)$ and write $k=gk'$ and $l=gl'$. Then the following hold.
  \begin{enumerate}
   \item Suppose $\zeta_{k'}<h(1)$. Then there is $s\in\Z$ such that $h(a)^{1/k}\zeta_{k}^s\in U_{b,l}$.
   \item Suppose $h(1)<\zeta_{k'}$. Then there is $s\in\Z$ such that $h(a)^{1/k}\zeta_{k}^s\in U_{b,l}$ if and only if 
    \[
     h(l'a+k'b)\in\big(h(-k'),1\big).
    \]
  \end{enumerate}
 \end{lemma}
 
 \begin{proof}
 Using Corollaries ~\ref{kthroot1} and \ref{kthroot2},  $h(a)^{1/k}\zeta_{k}^s\in U_{b,l}$ if and only if
  \[
  h(l'a)^{1/k'}\zeta_{k'}^{\floor{\frac{la}{r}}-l\floor{\frac{a}{r}}+sl'}\in\big(h(-b-1),h(-b)\big).
  \]
  Since $\gcd(k',l')=1$, there is $s\in\Z$ with  $h(a)^{1/k}\zeta_{k}^s\in U_{b,l}$ if and only if there is $t\in\Z$ with $h(l'a)^{1/k'}\zeta_{k'}^{t}\in\big(h(-b-1),h(-b)\big)$. 
  
\medskip\noindent
If $\zeta_{k'}<h(1)$, then there is such a $t$, regardless of what $a,b,l$ are. 

\smallskip\noindent
If $h(1)<\zeta_{k'}$, then we may use Proposition~\ref{kthpower} to conclude that there is $s\in\Z$ with $h(a)^{1/k}\zeta_{k}^s\in U_{b,l}$ if and only if 
\[
h(l'a)\in\big(h(-k'b-k'),h(-k'b)\big).
\]
After simplification, this means $h(l'a+k'b)\in \big(h(-k'),1\big)$. 
  
 \end{proof}

\medskip\noindent
This proof can be modified to prove the next analogous result.

 \begin{lemma}\label{V's}
 Let $a,b\in\Z$ and $k,l\in\N_+$.  Suppose $g=\gcd(k,l)$ and write $k=gk'$ and $l=gl'$. Then the following hold.
  \begin{enumerate}
   \item Suppose $\zeta_{k'}<h(-1)$. Then there is $s\in\Z$ such that $h(a)^{1/k}\zeta_{k}^s\in V_{b,l}$.
   \item Suppose $h(-1)<\zeta_{k'}$. Then there is $s\in\Z$ such that $h(a)^{1/k}\zeta_{k}^s\in V_{b,l}$ if and only if 
    \[
     h(l'a+k'b')\in\big[1,h(-k')\big].
    \]
  \end{enumerate}
 \end{lemma}

\medskip\noindent
Let $\vec a=(a_1,\dots,a_n)\in\Z^n$, $\vec k=(k_1,\dots,k_n)\in\N^n$, $I\subseteq[n]$ and $J\subseteq[n]\setminus I$ be given.  
 We define
  \[
  V_{\vec a,\vec k,J}=\bigcap_{j\in J}(V_{a_j,k_j}\setminus \tilde V_{a_j,k_j})\cap \bigcap_{j\in [n]\setminus(I\cup J)}\tilde V_{a_j,k_j}.
  \]
Note that $V_{a_j,k_j}\setminus \tilde V_{a_j,k_j}$ has $2 k_j$ many points. Therefore $V_{\vec a,\vec k,J}$ is finite for $J\neq \emptyset$ and $V_{\vec a,\vec k,\emptyset}$ is an open subset of $\S$.

\medskip\noindent
We record the following without proof.

\begin{lemma}\label{nonempty_finite}
Let $\vec a=(a_1,\dots,a_n)\in\Z^n$, $\vec k=(k_1,\dots,k_n)\in\N_+^n$, and $\emptyset\neq J\subseteq [n]$.  Then $h(c)\in V_{\vec a,\vec k, J}\cap\Gamma_r$ if and only if there is a subset $J'$ of $J$ such that $c=\frac{-a_j}{k_j}$ for every $j\in J'$, $c=\frac{-a_j-1}{k_j}$ for every $j\in J\setminus J'$, and $h(c)\in\bigcap_{j\in[n]\setminus J} \tilde V_{a_j,k_j}$.

\end{lemma}


\begin{definition}\label{pattern_def}
 Let $\vec k=(k_1,\dots,k_n)\in\Z^n$ and $I\subseteq[n]$. We say that $(a_1,\dots,a_n)\in\Z^n$ realize $(\vec k,I)$-{\it pattern } if there is $c\in\Z$ such that 
  \[
  a_i+k_ic\in P_r\Longleftrightarrow i\in I.
  \]
\end{definition}

 \medskip\noindent
For $\vec k=(k_1,\dots,k_n)\in\Z^n$, let $|\vec k|=(|k_1|,\dots,|k_n|)$. Then using Lemma~\ref{negations}, $(a_1,\dots,a_n)\in\Z^n$ realizes the $(\vec k,I)$-pattern if $(a_1',\dots,a_n')$ realizes the $(|\vec k|,I)$-pattern where $a_i'=a_i$ for $k_i\geq 0$ and $a_i'=-a_i-1$ for $k_i<0$.   Therefore, we may focus on the case that $\vec k\in\N^n$.


 \medskip\noindent
By $(\ref{5})$ and $(\ref{6})$, if $\vec k\in\N^n$, then $(a_1,\dots,a_n)\in\Z^n$ realize $(\vec k,I)$-pattern if and only if 
 \[
 \bigcap_{i\in I}U_{a_i,k_i}\cap\bigcap_{j\notin I}V_{a_j,k_j}\cap\Gamma_r\neq\emptyset.
 \]
 
 \medskip\noindent
 We may decompose the intersection above as 
  \[
  \left(\bigcap_{i\in I}U_{a_i,k_i}\cap \bigcap_{j\notin I}\tilde V_{a_j,k_j}\cap\Gamma_r\right)\cup\left( \bigcap_{i\in I}U_{a_i,k_i}\cap \bigcup_{\emptyset\neq J\subseteq[n]\setminus I} V _{\vec a,\vec k,J}\cap\Gamma_r\right).
  \]
Using Lemma~\ref{nonempty_finite}, the finite component is under control. 

\medskip\noindent
Let's focus on $\bigcap_{i\in I}U_{a_i,k_i}\cap V_{\vec a,\vec k,\emptyset}\cap\Gamma_r$.
Since $\Gamma_r$ is dense in $\S$, this set is nonempty if and only if the open set
 \[
  \bigcap_{i\in I}U_{a_i,k_i}\cap \bigcap_{j\notin I}\tilde V_{a_j,k_j}
 \] 
is nonempty. 

\medskip\noindent
Using Proposition~\ref{intersection_of_intervals}, it is easy to see that this intersection is nonempty if and only if one of the following holds:

 \begin{enumerate}
  \item there are $i_0\in I$ and $s\in\Z$ such that 
 
    \[h(-a_{i_0}-1)^{1/k_{i_0}}\zeta_{k_{i_0}}^s\in \bigcap_{i\in I,i\neq i_0}U_{a_i,k_i}\cap\bigcap_{j\notin I}\tilde V_{a_j,k_j},\]
  
  \item there are $j_0\in[n]\setminus I$ and $s\in\Z$ such that 
    \[h(-a_{j_0})^{1/k_{j_0}}\zeta_{k_{j_0}}^s\in \bigcap_{i\in I}U_{a_i,k_i}\cap\bigcap_{j\notin I,j\neq j_0}\tilde V_{a_j,k_j}.\]
 \end{enumerate}

\medskip\noindent
In order to summarize these observations, we make the following definitions: let $k,l\in\N$ with $g=\gcd(k,l)$ and $k':=k/g$, $l':=l/g$
 \[
 A_{k,l}:=
  \begin{cases}
 \Z\times\Z & :\text{if }\zeta_{k'}<h(1)\\ 
 \left\{(a,b)\in\Z\times\Z:h(l'a+k'b)\in\big(h(-k'),1\big)\right\}  & :\text{if }h(1)<\zeta_{k'}
  \end{cases} 
 \]
  
  \[
 B_{k,l}:=
  \begin{cases}
 \Z\times\Z & :\text{if }\zeta_{k'}<h(-1)\\ 
 \left\{(a,b)\in\Z\times\Z:h(l'a+k'b)\in\big(1,h(-k')\big)\right\}  & :\text{if }h(-1)<\zeta_{k'}
  \end{cases} 
 \]
 
%
%
%
%
 
\medskip\noindent
Combining Lemmas \ref{U's} and \ref{V's} with the observations above, we obtain the following. 

\begin{proposition}\label{patternQE}
Let $\vec k=(k_1,\dots,k_n)\in\N^n$ and $I\subseteq[n]$. Then $\vec{a}=(a_1,\dots,a_n)\in\Z^n$ realizes the $(\vec k,I)$-pattern if and only if one of the following conditions holds:
 \begin{enumerate}
  \item $\bigcap_{i\in I}U_{a_i,k_i}\cap \bigcup_{\emptyset\neq J\subseteq[n]\setminus I} V _{\vec a,\vec k,J}\cap\Gamma_r\neq\emptyset$.
  \item there is $i_0\in I$ with $(-a_{i_0}-1,a_i)\in A_{k_{i_0},k_i}$ for every $i\in I\setminus\{i_0\}$ and $(-a_{i_0}-1,a_j)\in B_{k_{i_0},k_j}$ for every $j\in [n]\setminus I$.
  \item there is $j_0\in I$ with $(-a_{j_0},a_i)\in A_{k_{j_0},k_i}$ for every $i\in I$ and $(-a_{j_0},a_j)\in B_{k_{j_0},k_j}$ for every $j\in [n]\setminus (I\cup\{j_0\})$.
 \end{enumerate}
\end{proposition}

\medskip\noindent
The final result of this section expresses the interval $\big(h(k),1\big)$ in terms of $P_r$ when $h(1)<\zeta_k$.

\begin{lemma}\label{cosets}
Let $k>0$ be such that $h(1)<\zeta_k$. Then $\alpha\in\big(1,h(k)\big)$ if and only if
 \[
 \alpha^{-1}\in\bigcup_{i=0}^{k-1}h(-i) P_r\cup\{h(-1),h(-2),\dots,h(-(k-1))\}.
 \]
 \emph{(Here $h(i)P_r$ is short for the interval $\big(h(-i-1),h(-i)\big)$.)}
\end{lemma}

\begin{proof}
The assumption $h(1)<\zeta_k$ gives $h(i-1)<h(i)$ for every $i\in[k]$. So we have the decomposition 
 \begin{align*}
  \big(1,h(k)\big)&=\big(1,h(1)\big]\cup\big(h(1),h(2)\big]\cup\cdots\cup\big(h(k-1),h(k)\big)\\
             &=\big(1,h(1)\big)\cup\cdots\cup\big(h(k-1),h(k)\big)\cup\{h(1),\dots,h(k-1)\}\\
             &=\bigcup_{i=0}^{k-1}h(i)\big(1,h(1))\cup\{h(1),\dots,h(k-1)\}
 \end{align*}
This finishes the proof, since $\beta\in\big(1,h(1))$ if and only if $\beta^{-1}\in P_r$.
\end{proof}

\section{Expanding the Group of Integers}\label{back_and_forth_section}
We would like to consider the model theoretic structure obtained by expanding the abelian group of integers by  the subset $P_r$. We have seen above that $\Z$ is isomorphic as an abelian group with a subgroup $\Gamma_r$ of $\S$ that happens to be dense in $\S$. The work in the previous section was mostly done in $\Gamma_r$, but it is straightforward to pull those results back to $\Z$ via the map $h$.

\medskip\noindent
Let $L:=\{+,-,0,c\}$ be the language of abelian groups with a distinguished element $c$.  Let $T$ be the theory of the $L$-structure $(\Z,+,-,0,1)$.
 

\medskip\noindent
We extend $L$ to $L_P:=L\cup\{P\}$ where $P$ is a unary relation symbol. Our main objective is to study the $L_P$-structure 
 \[
  \mathfrak Z:=\Big(\Z,+,-,0,1,P_r\Big).
 \] 

\medskip\noindent
For $\vec k=(k_1,\dots,k_n)\in\Z^n$ and $I\subseteq[n]$, we define $\phi_{\vec k,I}(x_1,\dots,x_n)$ to be the following $L_P$-formula:
 \[
 \exists y\left(\bigwedge_{i\in I}x_i+k_iy\in P\wedge\bigwedge_{j\notin I}x_j+k_jy\notin P\right).
 \]
Therefore for $(a_1,\dots,a_n)\in\Z^n$, we have  $\mathfrak Z\models\phi_{\vec k,I}(a_1,\dots,a_n)$
 if and only if $(a_1,\dots,a_n)$ realizes the $(\vec k,I)$-pattern. 
Then using Proposition~\ref{patternQE} and Lemma~\ref{cosets}, there is a quantifier-free $L_P$-formula $\psi_{\vec k,I}(x_1,\dots,x_n)$ such that
 \[
  \mathfrak Z\models\forall x_1\cdots\forall x_n\big(\phi_{\vec k,I}(x_1,\dots,x_n)\leftrightarrow \psi_{\vec k,I}(x_1,\dots,x_n)\big)
 \]

\medskip\noindent
Let $T_r$ be the $L_P$-theory extending $T$ by the condition above; namely for every $\vec k=(k_1,\dots,k_n)\in\Z^n$ and $I\subseteq[n]$, we add the following sentence as an axiom:
 \[
 \forall x_1\cdots\forall x_n\left(\phi_{\vec k,I}(x_1,\dots,x_n)\leftrightarrow \psi_{\vec k,I}(x_1,\dots,x_n)\right)
 \]

\medskip\noindent
We shall construct a back-and-forth system between certain substructures of models of $T_r$. 

\medskip\noindent
Let $\Cal M$ and $\Cal N$ be $\aleph_0$-saturated models of $T_r$. Let $S_\Cal M$ be the collection of countable $L_P$-substructures $\Cal M'$ of $\Cal M$ such that $M'$ is a pure subgroup of $M$.  We define $S_\Cal N$ in a similar way. 

 \medskip\noindent
 Note that the group $\Z$ has a copy in each member of $S_\Cal M$ and $S_\Cal N$ as the subgroup generated by the constant $c$. However, those copies may not be isomorphic as $L_P$-substructures of $\Cal M$ and $\Cal N$. 

\medskip\noindent
Let $\mathfrak B(\Cal M,\Cal N)$ be the collection of $L_P$-isomorphisms $f:\Cal M'\to\Cal N'$, where $\Cal M'\in S_\Cal M$ and $\Cal N'\in S_\Cal N$.

\begin{proposition}
For $\aleph_0$-saturated models $\Cal M$ and $\Cal N$  of $T_r$, the collection $\mathfrak B(\Cal M,\Cal N)$ is a back-and-forth system.
\end{proposition}

\begin{proof}
 Let $f:\Cal M'\to\Cal N'$ be in $\mathfrak B(\Cal M,\Cal N)$ and $\alpha\in M\setminus M'$. By symmetry, it suffices to extend $f$ to an element of $\mathfrak B(\Cal M,\Cal N)$ that contains $\alpha$ in its domain. 
 
 \medskip\noindent
 Let  $M''$ be the pure subgroup of $M$ generated by $M'$ and $\alpha$; namely: 
  \[
   M''=\<M'\cup\{\alpha\}\>_M:=\{\gamma\in M:m\gamma\in M'\oplus\Z\alpha\text{ for some }m>0\}.
  \]
 Also let $\Cal M''$ be the $L_P$-substructure of $\Cal M$ with the underlying set $M''$. We would like to extend $f$ to $\Cal M''$. That amounts to finding $\beta\in N$ with the following property: 

\medskip
$(*)$ For every $a\in M'$, $k\in\Z$, $n>0$, and $\gamma\in M$ if $a+k\alpha=n\gamma$, then

there is $\delta\in N$ such that $f(a)+k\beta=n\delta$, and 
   \[
   \gamma\in P \Longleftrightarrow \delta\in P.
   \]

\medskip\noindent
This condition without the last part just means that $\<M' \cup\{\alpha\}\>_M$ and $\<N'\cup\{\beta\}\>_N$ are isomorphic as groups. Since the reducts of $\Cal M$ and $\Cal N$ to $L$ are models of $T$, there is certainly such an element $\beta$ in $N$. So the point is to find $\beta$ in a way that that isomorphism of groups is indeed an $L_P$-isomorphism.

\medskip\noindent
By saturation, it suffices to find $\beta\in N$ satisfying a given finite fragment of $(*)$. So let $a_1,\dots,a_m\in M'$, $k_1,\dots,k_m\in\Z$, $n_1,\dots,n_m\in\N_+$, and $\gamma_1,\dots,\gamma_m\in M$ be such that $a_i+k_i\alpha=n_i\gamma_i$ for each $i$. Then we need to find $\beta,\delta_1,\dots,\delta_m\in N$ such that $f(a_i)+k_i\beta=n_i\delta_i$ and $\delta_i\in P$ if and only if $\gamma_i\in P$ for every $i$.

\medskip\noindent
Let $\nu=\lcm(n_1,\dots,n_m)$, and let $d\in\{0,\dots,\nu-1\}$ and $\alpha'\in M$ be such that $\alpha=d+\nu\alpha'$. Then $\gamma_i=a_i'+k_i'\alpha'$, where $n_i a_i'=a_i+k_id$ and $k_i'=\frac{k\nu}{n_i}$. Since $M'$ is pure in $M$, it contains $a_i'$. Therefore it suffices to find $\beta'\in N$ such that for every $i$:
 \[
 a_i'+k_i'\alpha'\in P \Longleftrightarrow f(a_i')+k_i'\beta'\in P
 \]
Taking $\vec{k}=(k_1',\dots,k_n')$ and $I=\{i:a_i'+k_i'\alpha'\in P\}$, we have 
 \[
 \Cal M\models \phi_{\vec{k},I}(a_1',\dots,a_n').
 \]
Hence 
\[
 \Cal M\models \psi_{\vec{k},I}(a_1',\dots,a_n') \text{ and }\Cal M'\models \psi_{\vec{k},I}(a_1',\dots,a_n').
 \]
Thus
\[
 \Cal N'\models \psi_{\vec{k},I}\left(f(a_1'),\dots,f(a_n')\right)\text{ and }\Cal N\models \psi_{\vec{k},I}\left(f(a_1'),\dots,f(a_n')\right).
 \] 
As a result $\Cal N\models \phi_{\vec{k},I}(f(a_1'),\dots,f(a_n'))$ and hence there is $\beta'\in N$ with the desired property:
\[
f(a_i')+k_i'\beta'\in P\Longleftrightarrow i\in I\Longleftrightarrow a_i'+k_i'\alpha'\in P.
 \]
\end{proof}

\subsection{Quantifier Elimination and Axiomatization}
The theory $T_r$ does not have quantifier elimination for the obvious reason that for any $n>1$, the definable subgroup consisting of elements divisible by $n$ is not quantifier-free definable. However, we still do not get quantifier elimination after adding predicate symbols to represent those subgroups, because we also need to know whether the element obtained by dividing by $n$ is in $P$ or not. So for every $n\geq 1$ we add two new unary predicate symbols $D_{n,+}$ and $D_{n,-}$ to the language $L_P$ to obtain $L_\pm$ and let $T_\pm$ be the definitional extension of $T_r$ to an $L_\pm$-theory by adding the following for each $n\geq 1$:
 \[
 \forall x\big(D_{n,+}(x)\leftrightarrow\exists y(x=ny\wedge y\in P)\big)
 \]
 \[
 \forall x\big(D_{n,-}(x)\leftrightarrow\exists y(x=ny\wedge y\notin P)\big)
 \]
Therefore, every model $\Cal M$ of $T_r$ expands to a model of $T_\pm$; we still denote this extension by $\Cal M$. Note that for a model $\Cal M$ of $T_\pm$, we have $D_{n,+}(\Cal M)\cup D_{n,-}(\Cal M)=n M$  for every $n\geq 1$ and $D_{1,+}(\Cal M)=P(\Cal M)$.

\medskip\noindent
Now we are ready to prove Theorem~\ref{th_2} in a  stronger form.

\begin{theorem}\label{QE}
The theory $T_\pm$ has quantifier elimination.
\end{theorem}

\begin{proof}
It suffices to prove the following: 
 
 $(\dagger)$ Let $\Cal M$ and $\Cal N$ be models of $T_\pm$ and $\Cal A$ a common finitely generated 
 
 $L_\pm$-substructure of $\Cal M$ and $\Cal N$.  Then $\Cal M\equiv_A\Cal N$. (This means that 
 
 $\Cal M$ and $\Cal N$ are elementarily equivalent as $L_\pm(A)$-structures )
 
 \smallskip\noindent
 (For why this is enough, see, for instance, Proposition 18.2 of \cite{Kirby_book}.)

\medskip\noindent
We may assume that $\Cal M$ and $\Cal N$ are  $\aleph_0$-saturated. Let
 \[
 M'=\<A\>_M \text{ and } N'=\<A\>_N.
 \] 
Clearly, $M'$ and $N'$ are isomorphic as abelian groups via a map extending the identity map on $A$. If $n\alpha=a$ for some $a\in A$ and $\alpha\in M$, then there is $\beta\in N$ with $n\beta=a$. Then the isomorphism sends $\alpha$ to $\beta$. But we also have $\alpha\in P(\Cal M)$ if and only if $\beta\in P(\Cal N)$, since either both $\Cal M$ and $\Cal N$ satisfy  $D_{n,+}(a)$ or they both satisfy $D_{n,-}(a)$. Therefore $M'$ and $N'$ are underlying sets of $L_\pm$-substructures $\Cal M'$ and $\Cal N'$ of $\Cal M$ and $\Cal N$ respectively, and they are isomorphic. Since $M'$ and $N'$ are countable and pure in $M$ and $N$,  that isomorphism is in $\mathfrak B(\Cal M,\Cal N)$. It follows that $\Cal M\equiv_{M'}\Cal N$ and in particular $\Cal M\equiv_{A}\Cal N$.
\end{proof}

\medskip\noindent
Given $\aleph_0$-saturated models $\Cal M$ and $\Cal N$ of $T_r$, we may still have that $\mathfrak{B}(\Cal M,\Cal N)=\emptyset$. So in order to get completeness  we extend $T_\pm$ to $T_{r}^*$ by adding \emph{$\Z$-axioms}: Given $k\in\N_+$ if $k\in P_r$, then we add the axiom $k\in P$, otherwise we add the axiom $k\notin P$. (Recall that $1$ is in the language, so $k=1+\dots+1$.) Clearly, $T_{r}^*$ still has quantifier elimination.

\medskip\noindent
With this extension, we get Theorem~\ref{th_3}.

\begin{theorem}
The theory $T_{r}^*$ is complete.
\end{theorem}  

\begin{proof}
Clearly, $\mathfrak Z$ is an algebraically prime model of $T_{r}^*$ . Since $T_{r}^*$ has quantifier elimination, we get that $T_{r}^*$ is complete.
\end{proof}

\medskip\noindent
{\it Question.} According to this theorem each $T_{r}^*$ is a completion of $T_r$. Is it correct that each completion of $T_r$ is given as the theory of an expansion of the group of integers by a Beatty Sequence?

\section{Definable Sets}\label{definable_sets_section}

When dealing with the definable sets,  we constantly switch between $\mathfrak Z$ and the isomorphic structure
\[
 \mathfrak G:=\big(\Gamma_r,\cdot,^{-1},1,h(1),P_r\big),
 \]
where $P_r$ denotes the interval $\big(h(-1),1\big)\cap\Gamma_r$ of $\Gamma_r$. 

\medskip\noindent
For a subset $X$ of $\Gamma_r^n$ and $m>0$ we let
 \[
 X^{(m)}:=\big\{(x_1^m,\dots,x_n^m):\vec x\in X\big\},\text{ and}\]
\[X^{1/m}:=\{\vec y\in\Gamma_r^n:(y_1^m,\dots,y_n^m)\in X\}.
 \]
 
 \medskip\noindent
For a subset $A$ of $\S$, we define $[A]$ to be the subgroup of $\S$ generated by $A$ and 
 \[
 \<A\>:=\left\{\alpha\in\S:\alpha^m\in[A]\text{ for some }m>0\right\}.
 \]
 
\medskip \noindent
 We collect some easy facts about these notions.
 
 \begin{lemma}\label{union_of_powers}
 Let $X,Y\subseteq\Gamma_r^n$ and $m>0$. 
  \begin{enumerate}
  \item $D_{m,+}^{1/m}=P_r$ and $P_r^{(m)}=D_{m,+}$.
  \item $(X^{(m)})^{1/m}=X$ and $(X^{1/m})^{(m)}=X\cap(\Gamma_r^n)^{(m)}$
   \item $(X\cap Y)^{1/m}=X^{1/m}\cap Y^{1/m}$, $(X\cup Y)^{1/m}=X^{1/m}\cup Y^{1/m}$, and $(\Gamma_r^n\setminus X)^{1/m}=\Gamma_r^n\setminus X^{1/m}$.
   \item For $\vec a\in(\{0,\dots,m-1\})^n$, let $X_{\vec a}=h(\vec a)X\cap(\Gamma_r^n)^{(m)}$. Then 
    \[
    X=\bigcup_{\vec a}h(-\vec a)X_{\vec a}.
    \]
  \end{enumerate}
 \end{lemma}

\medskip\noindent
Before analyzing definable sets in detail, we would like to clarify the connection with the paper \cite{Tran-Walsberg_expansions_of_Z}.

\begin{lemma}\label{convex_lemma}
 Let $m>0$ and let $X=(\alpha,\beta)\cap\Gamma_r^{(m)}$. Then any set $(\alpha\gamma_1,\beta\gamma_2)\cap\Gamma_r^{(m)}$ with $\gamma_1,\gamma_2\in\Gamma_r^{(m)}$ is definable in $(\Gamma_r,\cdot,X)$.
 \end{lemma}
 
 \begin{proof}
 It suffices to show that $Y=(\alpha,\beta\gamma)\cap\Gamma_r^{(m)}$ is definable in $(\Gamma_r,\cdot,X)$, where $\gamma=\gamma_2\gamma_1^{-1}$. 
 
 \smallskip\noindent
 Let $\lambda=\beta\alpha^{-1}$ and take $\delta\in(\lambda^{1/2},\lambda)\cap\Gamma_r^{(m)}$. Suppose that $n$ is the largest natural number such that $\gamma\delta^{-n}<\lambda$.
 
 
 \smallskip\noindent 
 Then $Y$ is the following union: 
  \[
 Y= X\cup\delta X\cup\dots\cup\delta^{n}X\cup ((\alpha\delta^{n+1},\beta\gamma)\cap\Gamma_r^{(m)}),
  \]
 Note that 
  \[(\alpha\delta^{n+1},\beta\gamma)\cap\Gamma_r^{(m)}=\delta^{n+1}(X\cap\gamma(\delta^{n+1})^{-1}X).\]
So $Y$ is definable in $(\Gamma_r,\cdot,X)$.  
 \end{proof}

%

\noindent
Letting $X=P_r$ in this lemma, we have the following consequence.

\begin{corollary}
Any interval of $\Gamma_r$ is definable in $\mathfrak G$.
\end{corollary}

\medskip\noindent
Therefore, the structures $\mathfrak G$  and $\big(\Gamma_r,\cdot,\Cal O\big)$ are interdefinable,  where $\Cal O$ is the restriction of the orientation of $\mathbb S$ to $\Gamma_r$. In \cite{Tran-Walsberg_expansions_of_Z}, the authors study the latter structure; or rather the pull-back of it under $h$. Even though some of our results are slightly finer than theirs, they also work out some stability theoretic properties of this structure. Most notably, they show that it is dp-minimal.  


\medskip\noindent
We appeal to the topology on $\Gamma_r^n$ induced from $\S^n$ in order to study definable sets. An open basis for that topology on $\Gamma_r$ consists of sets of the form $I\cap \Gamma_r$ where $I$ is an open interval of $\S$; below we refer to these sets as \emph{convex} sets. From now on we use the word \emph{interval} to mean \emph{open interval of $\Gamma_r$}; so an interval is a convex set whose end points are in $\Gamma_r$. 

\medskip\noindent
The quantifier elimination result in the previous section gives a very simple characterization of definable sets in a model $\Cal M$ of $T_{r}^*$:  They are Boolean combinations of sets defined by formulas of the form
 \[
 a+\vec k\vec x=0, D_{m,+}(a+\vec k\vec x)\text{,and } D_{m,-}(a+\vec k\vec x)
 \]
where $\vec x=(x_1,\dots, x_n)$ is a tuple of variables, $a\in M$ and $\vec k\in\Z^n$.

\medskip\noindent
It is simpler in the sense that the negations of formulas $D_{m,+}(a+\vec k\vec x)$ and $D_{m,-}(a+\vec k\vec x)$ define sets that are finite unions of sets defined by the same kind of formulas.  Also we would like to consider the formula 
\[D_{m,-}(a+\vec k\vec x)\wedge  a+\vec k\vec x\neq 0\wedge a+\vec k\vec x\neq -m\]
in the place of $D_{m,-}(a+\vec k\vec x)$. 

\medskip\noindent
So any definable set in $\mathfrak G$ is of the form
\begin{equation}\label{equation_general}
\bigcup_{i=1}^s\bigcap_{j=1}^{t_i}X_{ij}, \tag{$*$}
 \end{equation}
where each $X_{ij}$ is one of the following forms
 
 \begin{equation}\label{equation_eq}
 \{\vec x\in\Gamma_r^n:\vec x^{\vec k}h(a)=1\}\tag{A}
 \end{equation}

 \begin{equation}\label{equation_neq}
 \{\vec x\in\Gamma_r^n:\vec x^{\vec k}h(a)\neq 1\}\tag{B}
 \end{equation}

 \begin{equation}\label{equation_m+}
\{\vec x\in\Gamma_r^n:\vec x^{\vec k}h(a)\in P_r^{(m)}\}\tag{C}
 \end{equation}

  \begin{equation}\label{equation_m-}
\{\vec x\in\Gamma_r^n:\vec x^{\vec k}h(a)\in Q_r^{(m)}\}\tag{D}
 \end{equation}
 where $\vec k=(k_1,\dots,k_n)\in\Z^n\setminus\{\vec 0\}$, $a\in \Z$, $m>0$, $\vec x^{\vec k}:=x_1^{k_1}x_2^{k_2}\cdots x_n^{k_n}$, and $Q_r=(1,h(-1))\cap\Gamma_r$. 

\medskip\noindent
Note that a set of the form (A) is nonempty if and only if $\kappa|a$ where $\kappa$ is the greatest common divisor of the integers $k_i$. We refer to a finite intersection of sets of the form (A) as an \emph{affine subset} of $\Gamma_r^n$ provided that it is nonempty. Affine subsets of $\Gamma_r$ are singletons, and if $Y$ is an affine subset of $\Gamma_r^n$ with $n>1$, then there is a projection $\pi:\Gamma_r^n\to\Gamma_r^d$ with $d<n$ such that $\pi|_Y$ is injective.  Also for an affine subset $Y$, $\vec b\in\Z^n$, and $N>0$, the set $(h(\vec b)Y)^{1/N}$ is either empty or an affine subset.

\medskip\noindent
Suppose that $Y$ is a set of the form (C) or (D), $\vec b\in\Z^n$, and let $N\in\N_+$ be a multiple of $m$. Then the set $(h(\vec b)Y)^{1/N}$ is nonempty if and only if $m|a-\vec k\vec b$, and in that case 
 \[
 (h(\vec b)Y)^{1/N}=\Big\{\vec y:(\vec y^{\vec k})^{N/m}\in h(\frac{\vec k\vec b-a}{m})I\Big\},
 \]
where $I$ is one of $P_r$ or $Q_r$. As a result, it is an open subset of $\Gamma_r^n$.

\medskip\noindent
Putting these together we have the following result.

\begin{proposition}\label{open_affine_prop}
Let $X$ be definable in $\mathfrak G$. Then there is $N\in\N_+$ such that for every $\vec b\in\Z^n$ the set $(h(\vec b)X)^{1/N}$ is  a union of an open set and finitely many subsets of affine sets.
\end{proposition}

\begin{proof}
Write $X$ as in ($*$), and let $N$ be the lowest common multiple of the integers $m$ appearing in the formulas defining the sets $X_{ij}$. Given $i\in\{1,\dots,s\}$, if one of the sets $X_{ij}$ is of the form (A), then 
the set 
\[(h(\vec b)\bigcap_{j}X_{ij})^{1/N}=\bigcap_{j}(h(\vec b)X_{ij})^{1/N}\]
 is  contained in an affine set. Otherwise this set is open as noted above. 
\end{proof}

\medskip\noindent
Now we focus on unary definable sets with the aim of proving Theorem~\ref{th_1}. Combined with Proposition~\ref{kthpower}, Proposition~\ref{open_affine_prop} gives the following for unary subsets of $\Gamma_r$.

\begin{corollary}\label{def->convex1}
Suppose that $X\subseteq\Gamma_r$ is definable in $\mathfrak G$. Then there is $N\in\N_+$ such that for every $n\in\{0,1,\dots,N-1\}$, the set $(h(n)X)^{1/N}$ is a finite union of convex sets and singletons. Moreover, the end points of the convex sets are in $\<\Gamma_r\>$.
\end{corollary}

\begin{proof}
Affine subsets of $\Gamma_r$ are singletons and the open set appearing in $(h(n)X)^{1/N}$ is a certain positive Boolean combination of sets of the form:
\[ \Big\{y\in\Gamma_r:y^{kN}\neq h(-n)\Big\}.\]
and
\[ \Big\{y\in\Gamma_r:y^{kN/m}\in h(\frac{kb-n}{m})I\Big\}\]
where $I$ is one of $P_r$ or $Q_r$.

\smallskip\noindent
Using Proposition~\ref{kthpower}, if such a combination is not empty, then it is a finite union of convex sets whose end points are in $\<\Gamma_r\>$.
\end{proof}

\begin{definition}
Let $X\subseteq \Z$. We say that $X$ has the \emph{uniform gaps property} if there is $M\in\N_+$ such that 
 \[X\cap\{x+1,\dots,x+M\}\neq\emptyset\]
for every $x\in X$.
\end{definition}

\noindent
Clearly, $\emptyset$ has the uniform gaps property and it is the only finite set that has the uniform gaps property. The following is also clear.

\begin{lemma}\label{uniform_gaps_lemma}
 If $X,Y\subseteq\Z$ have the uniform gaps property. Then $X\cup Y$ has the uniform gaps property. 
\end{lemma}

\begin{lemma}\label{uniform_gaps_convex}
If $C\subseteq\Gamma_r$ is a convex set, then $X=h^{-1}(C)$ has the uniform gaps property.
\end{lemma}

\begin{proof}
Let $C=(\alpha,\beta)\cap\Gamma_r$ and let $\gamma:=l(C)$. Take $m,n\in\N_+$ such that $h(m)<\gamma^{1/2}$ and $h(-n)<\gamma^{1/2}$. We claim that $M=\max\{m,n\}$ witnesses that $X$ has the uniform gaps property.

\medskip\noindent
Suppose that $x\in X$. If $h(x)<\alpha\gamma^{1/2}$, then $h(x)h(m)\in C$ and hence $x+m\in X$, and if $\alpha\gamma^{1/2}<h(x)$, then $h(x)h(n)\in C$ and hence $x+n\in X$.
\end{proof}

\medskip\noindent
\emph{Remark. }This proof does not work for larger models of $T_{r}^*$, because the length of $C$ might be infinitesimal with respect to $\Gamma_r$ and hence we cannot find suitable $m$ and $n$.

\medskip\noindent
The following is Theorem~\ref{th_1} from the Introduction. 

\begin{theorem}\label{uniform_gaps_th}
 Every infinite subset of $\Z$ definable in $\mathfrak Z$ has the uniform gaps property.
\end{theorem}

\begin{proof}
Let $X\subseteq \Z$ be definable in $\mathfrak Z$ and let $Y:=h(X)$.  

\medskip\noindent
By Corollary~\ref{def->convex1}, there is $N>0$ such that $(h(n)Y)^{1/N}$ is a finite union of convex sets and singletons for each $n\in\{0,1,\dots,N-1\}$.

\medskip\noindent
Using the second and the last part of Lemma~\ref{union_of_powers}
 \[
 Y=\bigcup_{n=0}^{N-1}h(-n)((h(n)Y)^{1/N})^{(N)}.
 \]
Therefore by Lemma~\ref{uniform_gaps_lemma}, it suffices to show that $h^{-1}\big(((h(n)Y)^{1/N})^{(N)}\big)$ has the uniform gaps property for each $n$. Since $y\in ((h(n)Y)^{1/N})^{(N)}$ if and only if $y=z^N$ for some $z\in (h(n)Y)^{1/N}$, it suffices to show that $h^{-1}\big((h(n)Y)^{1/N}\big)$ has the uniform gaps property. It is indeed the case using Lemmas \ref{uniform_gaps_lemma} and \ref{uniform_gaps_convex}.
\end{proof}

\begin{corollary}\label{reverse_gaps_cor}
Let  $X$ be an infinite subset of $\Z$ that is definable in $\mathfrak Z$. Then there is $M\in\N_+$ such that for every $x\in X$, the intersection $X\cap\{x-1,x-2,\dots,x-M\}$ is nonempty.
\end{corollary}

\begin{proof}
The set $-X:=\{y\in\Z:-y\in X\}$ is also definable in $\mathfrak Z$ and hence has the uniform gaps property. That translates to $X$ as the conclusion of the corollary.
\end{proof}

\begin{corollary}\label{ordering_cor}
Ordering of $\Z$ is not definable in $\mathfrak Z$.
\end{corollary}

\begin{proof}
If the ordering were definable in $\mathfrak Z$, then so would be the set of positive elements. However, $\N_+$ does not satisfy the conclusion of Corollary~\ref{reverse_gaps_cor}.
\end{proof}

\begin{corollary}\label{multiplication_cor}
Multiplication is not definable in $\mathfrak Z$.
\end{corollary}

\begin{proof}
If multiplication on $\Z$ is definable, then the ordering of $\Z$ is also definable using Lagrange's four-square theorem. 
\end{proof}

\medskip\noindent
We need the following lemma to show that $\mathfrak Z$ is unstable.

\begin{lemma}
Let $n>0$, then there is $m>0$ such that $m,2m,\dots,nm\in P_r$ and $-m,-2m,\dots,-nm\notin P_r$.
\end{lemma}

\begin{proof}
Let $\alpha=\min\{h(1),h(-1)\}$ and $\beta=\alpha^{\frac{1}{n+1}}$. By regular density of $\Gamma_r$ take $m\in\N_+$ such that $\beta^{-1}<h(m)$. Clearly, this $m$ satisfies the conclusion of the lemma.
\end{proof}

\begin{proposition}
The theory $T_{r}^*$ is not stable.
\end{proposition}

\begin{proof}
Let $\phi(x;y)$ be the $L_P$-formula $y-x\in P$. We show that this formula is unstable. For this, it suffices to prove the following: For every $n>0$, there are $a_1,\dots,a_n\in\Z$ such that 
 \[
 \mathfrak Z\models \phi(a_i,a_j)\Longleftrightarrow i<j.
 \]
 So let $n>0$ be given. Take $a_i=im$ where $m=m(n)$ is as in the lemma above.  Now $a_j-a_i=(j-i)m$ and hence $\phi(a_i,a_j)$ holds in $\mathfrak Z$ if and only if $i<j$.
 
\end{proof}
\noindent (We would like to thank Haydar G\"oral for the idea of this proof.)

\section{No Reducts}\label{no_reducts_section}
We prove that there are no intermediate structures between $(\Z,+)$ and $(\Z,+,P_r)$; actually we prove the same result for $(\Gamma_r,\cdot)$ and $(\Gamma_r,\cdot,P_r)$. 

\medskip\noindent
We first consider a subset $X$ of $\Gamma_r$ that is definable in $\mathfrak G$. 

\begin{proposition}\label{def->convex}
Let $X\subseteq\Gamma_r$ be definable in $\mathfrak G$. Suppose that $X$ is not definable in $(\Gamma_r,\cdot)$. Then there is $N>0$ and $L\in\{0,1,\dots,N-1\}$ such that both the set $(h(L)X)^{1/N}$ and its complement contain a convex set. \end{proposition}

\begin{proof}
By Corollary~\ref{def->convex1}, there is $N>0$ such that for each $n$, the set $(h(n)X)^{1/N}$ is a finite union of convex sets with end points in $\<\Gamma_r\>$, and singletons. Write 
    \[
    X=\bigcup_{n=0}^{N-1}h(-n)(h(n)X\cap\Gamma_r^{(N)}).
    \]
For each $n$, raising to power $N$ is a bijection between the sets $(h(n)X)^{1/N}$ and $h(n)X\cap\Gamma_r^{(N)}$.  Since $X$ is not definable in $(\Gamma_r,\cdot)$, there is $L\in\{0,\dots,N-1\}$  such that $(h(L)X)^{1/N}$ is infinite. If  $(h(L)X)^{1/N}$ is cofinite in $\Gamma_r$, then so is $h(L)X\cap\Gamma_r^{(N)}$. Once again, since $X$ is not definable in $(\Gamma_r,\cdot)$, $h(L)X\cap\Gamma_r^{(N)}$ cannot be cofinite in $\Gamma_r$. Therefore both $(h(L)X)^{1/N}$ and its complement contains a convex set.
\end{proof}

\medskip\noindent
So if $N$ and $L$ are as in this proposition, both $(h(L)X)^{1/N}$ and its complement are finite unions of convex sets and singletons, but neither is finite. Moreover, by Corollary~\ref{def->convex1}, the end points of the convex sets appearing in either union are in $\<\Gamma_r\>$. Therefore there is a finite union of convex sets with end points in $\<\Gamma_r\>$ that is definable in $(\Gamma_r,\cdot,X)$ and is not cofinite in $\Gamma_r$.

\begin{lemma}\label{powers_to_intervals}
Let $X=(\alpha,\beta)\cap\Gamma_r^{(m)}$ and suppose that $\gamma:=\alpha\beta\in\Gamma_r$. 
Then a convex set is definable in $(\Gamma_r,\cdot,X)$. Moreover, the end points of that convex set are in $\<\Gamma_r\cup\{\alpha,\beta\}\>$.
\end{lemma}

\begin{proof}

\medskip\noindent
Let $0\leq n<m$ be such that $\gamma\in h(n)\Gamma_r^{(m)}$. 

\medskip\noindent
We first consider the case that $n\neq 0$. Let $k:=\gcd(n,m)$, $n'=n/k$, and $m'=m/k$. So $\gamma=h(n)\gamma_0^m=(h(n')\gamma_0^{m'})^k$ for some $\gamma_0\in\Gamma_r$. 

\medskip\noindent
Let $l\in\Z$ be such that $h(n')\gamma_0^{m'}\in(\zeta_k^l,\zeta_k^{l+1})$. Also let 
\[\delta_0\in((\beta^{1/k})^{-1}\zeta_k^{-l},(\beta^{1/k})^{-1}\zeta_k^{-l+1})\cap\gamma_0^{-m'}\Gamma_r^{(m)}\text{, and }\delta=\delta_0^k.\]
By Lemma~\ref{convex_lemma}, the set $Y:=(\alpha,\beta\delta)\cap\Gamma_r^{(m)}$ is definable in $(\Gamma_r,\cdot,X)$ and using Lemma~\ref{kthroot_roots_of_unity} we have that $\gamma_Y:=\alpha^{1/k}(\beta\delta)^{1/k}\in h(n')\Gamma_r^{(m)}$. In particular, $\{1,\gamma_Y,\gamma_Y^2,\dots,\gamma_Y^{m-1}\}$ is a full set of representatives of cosets of $\Gamma_r^{(m)}$ in $\Gamma_r$.

\medskip\noindent
By Proposition~\ref{kthpower} 
\[
Y^{1/k}=\bigcup_{s=0}^{k-1}\zeta_k^s(\alpha^{1/k},(\beta\delta)^{1/k})\cap\Gamma_r^{(m')}.
\]
Take $\tau_1,\tau_2\in\Gamma_r^{(m)}$ such that $(\alpha^{1/k},(\beta\delta)^{1/k})\subseteq(\alpha\tau_1,\beta\tau_2)$ and $l((\alpha\tau_1,\beta\tau_2))<\zeta_k$. The set $ (\alpha\tau_1,\beta\tau_2)\cap\Gamma_r^{(m)}$ is definable in $(\Gamma_r,\cdot,X)$ by Lemma~\ref{convex_lemma} and hence
\[
Z:=((\alpha\tau_1,\beta\tau_2)\cap\Gamma_r^{(m)})\cap Y^{1/k}=(\alpha^{1/k},(\beta\delta)^{1/k})\cap\Gamma_r^{(m)}.
\]
is also definable in $(\Gamma_r,\cdot,X)$.

\medskip\noindent
Let $f:\Gamma_r\to\Gamma_r$ be defined as $f(x)=x^{-1}\gamma_Y$. It is clear that 
\[x\in(\alpha^{1/k},(\beta\delta)^{1/k})\Longleftrightarrow f(x)\in(\alpha^{1/k},(\beta\delta)^{1/k}).\]
Therefore $f^i(Z)=(\alpha^{1/k},(\beta\delta)^{1/k})\cap\gamma_Y^i\Gamma^{(m)}$ for $i=0,1,\dots,m-1$.  Hence the proper convex set
 \[
(\alpha^{1/k},(\beta\delta)^{1/k})\cap\Gamma_r=Z\cup f(Z)\cup\dots\cup f^{m-1}(Z)
 \]
is definable in $(\Gamma_r,\cdot,X)$.

\medskip\noindent
Now let $n=0$ and take $k>0$ maximum such that $\gamma\in\Gamma_r^{(m^k)}$; say $\gamma=\gamma_0^{m^k}$ with $\gamma_0\in(\zeta_{m^k}^l,\zeta_{m^k}^{l+1})\cap\Gamma_r$.

\medskip\noindent
Let 
 \[\delta_0\in((\beta^{1/m^k})^{-1}\zeta_{m^k}^{-l},(\beta^{1/m^k})^{-1}\zeta_{m^k}^{-l+1})\cap(\Gamma_r\setminus \gamma_0^{-1}\Gamma_r^{(m)})\text{ and }\delta:=\delta_0^{m^k}.\]
As in the previous case, we have $Y=(\alpha,\beta\delta)\cap\Gamma_r^{(m)}$ is definable in $(\Gamma_r,\cdot,X)$ and 
\[
Y^{1/m^k}=\bigcup_{s=0}^{m^k-1}\zeta_{m^k}^s(\alpha^{1/m^k},(\beta\delta)^{1/m^k})\cap\Gamma_r.
\]
Again, there are $\tau_1,\tau_2\in\Gamma^{(m)}$ such that
\[
Z:=((\alpha\tau_1,\beta\tau_2)\cap\Gamma_r^{(m)})\cap Y^{1/m^k}=(\alpha^{1/m^k},(\beta\delta)^{1/m^k})\cap\Gamma_r^{(m)}.
\]
Now we have $\gamma^*:=\alpha^{1/m^k}(\beta\delta)^{1/m^k}=\gamma_0\delta_0$ is not in $\Gamma_r^{(m)}$. Therefore using the previous case a proper convex set is definable in $(\Gamma_r,\cdot,Z)$, hence in $(\Gamma_r,\cdot,X)$.

\medskip\noindent
It follows from the proof that the end points of the proper convex set are in $\<\Gamma_r\cup\{\alpha,\beta\}\>$.

\end{proof}

%
%
%

\begin{lemma}\label{powers_to_interval}
Let $X=(\alpha,\beta)\cap\Gamma_r$ and $m>0$ with $l(\alpha,\beta)<\zeta_m$. Then for  $0<k\leq m$:
 \[
\{x_1x_2\cdots x_k: x_1,x_2,\dots,x_k\in X\}= (\alpha^k,\beta^k)\cap\Gamma_r
 \]
\end{lemma}

\begin{proof}
Let
 \[
 Y_k:=\{x_1x_2\cdots x_k: x_1,x_2,\dots,x_k\in X\}.
 \]
It is clear from the assumptions that $Y_k\subseteq (\alpha^k,\beta^k)\cap\Gamma_r$.  We prove equality by induction on $k$. 

\medskip\noindent
The case $k=1$ is trivial, so let $k>1$ and suppose that $Y_{k-1}=(\alpha^{k-1},\beta^{k-1})\cap\Gamma_r$. Then 
 \begin{align*}
 Y_k&=\bigcup_{x_k\in X}Y_{k-1}\, x_k\\
    &=\bigcup_{x_k\in X}((\alpha^{k-1},\beta^{k-1})\cap\Gamma_r)x_k\\
    &=\bigcup_{x_k\in X}(x_k\alpha^{k-1},x_k\beta^{k-1})\cap\Gamma_r.
 \end{align*}
Clearly, the last union is a convex set. Since $\alpha^k$ and $\beta^k$ are limit points of $Y_k$ (in $\S$), we get 
 \[
Y_k=(\alpha^k,\beta^k)\cap\Gamma_r.
 \]
\end{proof}

\begin{corollary}\label{powers_to_interval_cor}
Let $X=(\alpha,\beta)\cap\Gamma_r$, where $\alpha,\beta\in\<\Gamma_r\>$. Then an interval is definable in $(\Gamma_r,\cdot,X)$.
\end{corollary}

\begin{proof}
Let $N\in\N_+$ such that $\alpha^N,\beta^N\in\Gamma_r$, and take $\gamma\in\Gamma_r$ such that $l((\alpha,\beta\gamma))<\zeta_N$. Then $(\alpha^N,\beta^N\gamma^N)\cap\Gamma_r$ is an interval and is definable in $(\Gamma_r,\cdot,X)$ by Lemma~\ref{powers_to_interval}.
\end{proof}

%

\begin{lemma}\label{convex_to_interval}
Let $X$ be a finite union of convex sets whose end points are in $\<\Gamma_r\>$.  If $X$ is not cofinite in $\Gamma_r$, then an interval is definable in $(\Gamma_r,\cdot,X)$.
\end{lemma}

\begin{proof}

\medskip\noindent
After translating, we may assume that $1\notin X$. 

\medskip\noindent
Let $X=\bigcup_{i=1}^m C_i$ where  $C_i=(\alpha_i,\beta_i)\cap\Gamma_r$ with 
 \[\alpha_1<\beta_1\leq\alpha_2<\beta_2\leq\dots\leq\alpha_m<\beta_m.\]
We prove the result by induction on $m$. If $m=1$, then there is nothing to do. So suppose $m>1$.

\medskip\noindent
Put $\lambda_j=\beta_j\alpha_j^{-1}$ for all $j$, and $\delta_j=\alpha_{j+1}\beta_j^{-1}$  for $j=1,\dots,m$, where $\alpha_{m+1}:=\alpha_1$.
 
\medskip\noindent
Below we consider three cases and in each one we construct a convex set definable in $(\Gamma_r,\cdot,X)$. Moreover the end points of those convex sets will still be in $\<\Gamma_r\>$. Hence by using Corollary~\ref{powers_to_interval_cor}, we will be done.
 
\medskip\noindent
  \underline{Case 1:} The set $\{\delta_1,\dots,\delta_m\}$ is not a singleton.
  
\noindent
Let $\delta_{j_0}$ and $\delta_{j_{1}}$ be the two smallest elements of this set and let $\lambda=\min\{\delta_{j_1}\delta_{j_0}^{-1},\lambda_1,\dots,\lambda_m\}$. Take $\gamma\in(\lambda^{1/2},\lambda)\cap\Gamma_r$ and let $N>0$ be such that 
 \[\gamma<\gamma^2<\cdots<\gamma^{N-1}<\delta_{j_0}<\gamma^{N}.\] 
Note that $\gamma^{N}<\delta_{j_1}$ and hence the union
 \[
 Y:=X\cup \gamma X \cup\gamma^2 X\cup\dots\cup\gamma^N X
 \]
 is a union of at most $m-1$ convex sets, yet it has infinite complement in $\Gamma_r$. So by induction hypothesis, there is a nonempty proper convex subset $C$ of $\Gamma$ that is definable in $(\Gamma_r,\cdot,Y)$. Since $Y$ is definable in $(\Gamma_r,\cdot,X)$, the set $C$ is definable in $(\Gamma_r,\cdot,X)$.
 
 \medskip\noindent
  \underline{Case 2:} The set $\{\delta_1,\dots,\delta_m\}$ is a singleton, but the set $\{\lambda_1,\dots,\lambda_m\}$ is not a singleton.
 
  \noindent
  The set $Z=\Gamma_r\setminus X$ is a set as in Case 1. So a nonempty proper convex subset of $\Gamma$ is definable in $(\Gamma_r,\cdot,Z)$, hence in $(\Gamma_r,\cdot,X)$.

 \medskip\noindent
  \underline{Case 3:} Both $\{\delta_1,\dots,\delta_m\}$  and $\{\lambda_1,\dots,\lambda_m\}$ are singletons.
 
  \noindent
Note that we need to have $\delta_j\lambda_j=\zeta_m$ for each $j$. Hence $\alpha_{j+1}=\zeta_m\alpha_j$ and $\beta_{j+1}=\zeta_m\beta_j$ for each $j=1,\dots,m-1$. Therefore
 \[
 X=\big(\bigcup_{j=0}^{m-1}\zeta_m^j(\alpha_1,\beta_1)\big)\cap\Gamma_r.
 \]
So $X^{(m)}=(\alpha_1^m,\beta_1^m)\cap\Gamma^{(m)}$. Therefore a convex set is definable in $(\Gamma_r,\cdot,X)$ by Lemma~\ref{powers_to_intervals}.
\end{proof}

%
 

\medskip\noindent
Putting Proposition~\ref{def->convex} and Lemma~\ref{convex_to_interval} together, we get the desired result for unary definable sets.

\begin{proposition}\label{unary_reduct}
Let $X\subseteq\Gamma_r$ be definable in $\mathfrak G$, but not in $(\Gamma_r,\cdot)$. Then $(\Gamma_r,\cdot,X)$ is interdefinable with $\mathfrak G$.
\end{proposition}

\medskip\noindent
Now we handle the general case.

\begin{theorem}\label{no_reducts}
Let $X\subseteq\Gamma_r^n$ be definable in $\mathfrak G$, but not in $(\Gamma_r,\cdot)$. Then $(\Gamma_r,\cdot,X)$ is interdefinable with $\mathfrak G$.
\end{theorem}

\begin{proof}
We proceed by induction on $n$. Proposition~\ref{unary_reduct} serves as the case $n=1$. So let's assume that $n>1$ and that the result holds for $m<n$.

\medskip\noindent
So let $X\subseteq\Gamma_r^n$ be definable in $\mathfrak G$, and take $N\in\N_+$ such that for every $\vec a\in\Z^n$, the set  $(h(\vec a)X)^{1/N}$ is a union of an open set and finitely many sets contained in affine sets. Write 
 \[
 X=\bigcup_{\vec a\in A}h(-\vec a)(h(\vec a)X\cap(\Gamma_r^{(N)})^n),
 \]
 where $A$ is a finite subset of $\Z^n$.  It is easy to see that for each $\vec a$, the set $(h(\vec a)X)^{1/N}$ has empty interior if and only if $h(\vec a)X\cap(\Gamma_r^{(N)})^n$ is contained in a finite union of affine sets.

\medskip\noindent
First suppose that $(h(\vec a)X)^{1/N}$ has empty interior for each $\vec a\in A$. Then $(h(\vec a)X)^{1/N}$ is a finite union of sets contained in an affine set and each of those sets is in bijection with a subset of $\Gamma_r^d$ for some $d<n$ via a projection. If all of those subsets of $\Gamma_r^d$ are definable in $(\Gamma_r,\cdot)$, then so is $(h(\vec a)X)^{1/N}$. This cannot be correct for all $\vec a$, because $X$ is not definable in $(\Gamma_r,\cdot)$. Hence we obtain a subset of $\Gamma_r^d$ that is not definable in $(\Gamma_r,\cdot)$, but is definable in $(\Gamma_r,\cdot,X)$. So we obtain the result using the induction hypothesis. 

\medskip\noindent
So we assume that the set $A^*$ of $\vec a\in A$ such that $(h(\vec a)X)^{1/N}$ has nonempty interior is nonempty.  Hence the set
 \[
 X^*=X\bigcup_{\vec a\in  A^*}{h(-\vec a)(h(\vec a) X\cap(\Gamma^{(N)})^n)}
 \]
is definable in $(\Gamma_r,\cdot,X)$ and $(h(\vec a)X^*)^{1/N}$ has nonempty interior for each $\vec a\in A^*$. Put $Y:=\Gamma_r\setminus X^*$. If $(h(\vec a)Y)^{1/N}$ has empty interior for each $\vec a\in A^*$, then we may proceed as in the previous case. Therefore we may assume that $(h(\vec a)Y)^{1/N}$ has nonempty interior for some $\vec a\in A^*$.

\medskip\noindent
For $\vec\gamma\in\Gamma_r^{n-1}$, consider the sets
 \[
  X^*_{\vec\gamma}:=\{\delta\in\Gamma_r:(\vec\gamma,\delta)\in (h(\vec a)X^*)^{1/N}\},
 \]
 \[
  Y_{\vec\gamma}:=\{\delta\in\Gamma_r:(\vec\gamma,\delta)\in (h(\vec a)Y)^{1/N}\}
 \]
For each $\vec\gamma$, these sets are either empty or has nonempty interior. If there is $\vec\gamma\in\Gamma_r^{n-1}$ such that both $X^*_{\vec\gamma}$ and $Y_{\vec\gamma}$ have nonempty interior, then $X^*_{\vec\gamma}$ is not definable in $(\Gamma_r,\cdot)$, and we are done using the induction hypothesis. Otherwise for each $\vec\gamma\in\Gamma_r^{n-1}$, the set $X^*_{\vec\gamma}$ is either empty or is $\Gamma_r$. Then the set $Z:=\{\vec\gamma:X^*_{\vec\gamma}=\Gamma_r\}$ is a  subset of $\Gamma_r^{n-1}$ definable in $(\Gamma_r,\cdot,X)$. If $Z$ is definable in $(\Gamma_r,\cdot)$, then so are $(h(\vec a)X^*)^{1/N}$ and $X$. Therefore $Z$ is not definable in $(\Gamma_r,\cdot)$ and once again we are done by the induction hypothesis.

\end{proof}

\end{document}